\UseRawInputEncoding
\documentclass[leqno]{amsart}

\usepackage{stmaryrd,graphicx}
\usepackage{amssymb,mathrsfs,amsmath,amscd,amsthm,color}
\usepackage{float}
\usepackage{mathabx}
\usepackage[all,cmtip]{xy}
\DeclareMathAlphabet{\mathpzc}{OT1}{pzc}{m}{it}
\usepackage{amsfonts,latexsym,wasysym}

\newcommand{\supp}{\text{supp}}

\newcommand{\sw}{\bigcurlyvee}

\newcommand{\ds}{\displaystyle}
\newcommand{\ui}{I}

\newcommand{\mcg}{\mathcal{G}}

\newcommand{\mcw}{\mathcal{W}}

\newcommand{\bbe}{\mathbb{E}}

\newcommand{\bbn}{\mathbb{N}}

\newcommand{\bbr}{\mathbb{R}}

\newcommand{\bbz}{\mathbb{Z}}

\newcommand{\im}{\operatorname{Im}}

\newtheorem{theorem}{Theorem}[section]
\newtheorem{lemma}[theorem]{Lemma}
\newtheorem{proposition}[theorem]{Proposition}
\newtheorem{corollary}[theorem]{Corollary}
\theoremstyle{definition}\newtheorem{definition}[theorem]{Definition}
\newtheorem{example}[theorem]{Example}

\newtheorem{remark}[theorem]{Remark}

\begin{document}
\title{The \v{C}ech homotopy groups of a shrinking wedge of spheres}
\author[J. Brazas]{Jeremy Brazas}
\address{West Chester University\\ Department of Mathematics\\
West Chester, PA 19383, USA}
\email{jbrazas@wcupa.edu}


\date{\today}

\begin{abstract}
We compute the \v{C}ech homotopy groups of the $m$-dimensional infinite earring space $\mathbb{E}_m$, i.e. a shrinking wedge of $m$-spheres. In particular, for all $n,m\geq 2$, we prove that $\check{\pi}_n(\mathbb{E}_m)$ is isomorphic to a direct sum of countable powers of homotopy groups of spheres: $\bigoplus_{1\leq j\leq \frac{n-1}{m-1}}\left(\pi_{n}(S^{mj-j+1})\right)^{\mathbb{N}}$. Equipped with this isomorphism and infinite-sum algebra, we also construct new elements of $\pi_n(\mathbb{E}_m)$ with a view toward characterizing the image of the canonical homomorphism $\Psi_{n}:\pi_n(\mathbb{E}_m)\to \check{\pi}_{n}(\mathbb{E}_m)$. We prove that $\Psi_{n}$ is a split epimorphism when $n\leq 2m-1$ and we identify a candidate for the image of $\Psi_n$ when $n>2m-1$.
\end{abstract}

\maketitle

\section{Introduction}

In 1962, Barratt and Milnor \cite{BarrattMilnor} proved that the rational homology groups of a shrinking wedge of $m$-spheres  ($m\geq 2$), which we refer to as the \textit{$m$-dimensional earring space} $\bbe_m$, are uncountable in arbitrarily high dimensions. In 2000, Eda and Kawamura proved that $\bbe_m$ is $(m-1)$-connected and locally $(m-1)$-connected and that $\pi_m(\bbe_m)$ is canonically isomorphic to the Baer-Specker group $\bbz^{\bbn}$. Built into the proof of Eda-Kawamura's Theorem is the striking fact that higher homotopy groups are ``infinitely commutative" in a strong sense. 

It remains an open problem to give a full description of the groups $\pi_n(\bbe_m)$, $n>m$ in the spirit of the Hilton-Milnor Theorem. In this paper we give a complete description of the \v{C}ech homotopy groups $\check{\pi}_{n}(\bbe_m)$ in terms of the homotopy groups of spheres for all $n,m\geq 2$. We build off of the computation of $\check{\pi}_{m+1}(\bbe_m)$, $m\geq 2$ established by K. Kawamura in \cite{Kawamurasuspensions}. Our first main result is the following.

\begin{theorem}\label{thm1}
For all $m,n\geq 2$, there is a isomorphism \[\check{\pi}_n(\bbe_m)\cong \bigoplus_{1\leq j\leq \frac{n-1}{m-1}}\left(\pi_{n}(S^{mj-j+1})\right)^{\bbn}.\]
\end{theorem}

The isomorphism in Theorem \ref{thm1} is not canonical but depends only an enumeration of the union of coherently chosen Hall bases $H_k$ on finite sets $A_k=\{a_1,\dots,a_k\}$. The upper bound $\frac{n-1}{m-1}$ on the index $j$ simply indicates when the homotopy groups of spheres might be non-trivial (if $j>\frac{n-1}{m-1}$, then $n<mj-j+1$ and $\pi_{n}(S^{mj-j+1})=0$). Low-dimensional cases of special interest are highlighted in Example \ref{specialcaseexample}.

When $n>m$, the group $\pi_n(\bbe_m)$ is potentially more complicated than $\check{\pi}_n(\bbe_m)$. However, there is a canonical homomorphism $\Psi_n:\pi_n(\bbe_m)\to \check{\pi}_n(\bbe_m)$ relating the two. Using the isomorphism in Theorem \ref{thm1} as a book-keeping device, we construct and distinguish new elements of $\pi_n(\bbe_m)$ in an effort to characterize the image of $\Psi_n$. We do not attempt to address whether or not $\Psi_n$ is injective. Our second theorem establishes the surjectivity of $\Psi_n$ in a certain range of dimensions.

\begin{theorem}\label{thm2}
If $m,n\geq 2$ and $n\leq 2m-1$, then the canonical homomorphism $\Psi_{n}:\pi_n(\bbe_m)\to \check{\pi}_{n}(\bbe_m)$ is a split epimorphism.
\end{theorem}

Theorem \ref{thm2} is proved in two cases. The stable range ($n<2m-1$) is proved in Corollary \ref{firstcasecor} where $\check{\pi}_{n}(\bbe_m)\cong \pi_{n}(S^m)^{\bbn}$. The case $n=2m-1$ is proved in Lemma \ref{edgelemma} where $\check{\pi}_{2m-1}(\bbe_m)\cong \pi_{2m-1}(S^{m})^{\bbn}\oplus \bbz^{\bbn}$. The image of the splitting homomorphism $\check{\pi}_{2m-1}(\bbe_m)\to \pi_{2m-1}(\bbe_m)$ constructed in Lemma \ref{edgelemma} meets the kernel of the homomorphism induced by the canonical embedding $\sigma:\bbe_m\to \prod_{j=1}^{\infty}S^m$ precisely in the subgroup $\mcw_{2n-1}(\bbe_n)$ that is computed in \cite{BrazasWHProducts}. Moreover, the author has suggested in \cite[Conjecture 1.3]{BrazasWHProducts} that $\mcw_{2n-1}(\bbe_n)$ is equal to $\ker(\sigma_{\#})$. It follows from the standard splitting of $\sigma_{\#}$ reviewed in Section \ref{sectionshrinkingwedge} that \cite[Conjecture 1.3]{BrazasWHProducts} is equivalent to conjecturing that $\Psi_{2m-1}:\pi_{2m-1}(\bbe_m)\to \check{\pi}_{2m-1}(\bbe_m)$ is an isomorphism for all $m\geq 2$.

Although the surjectivity of $\Psi_{n}$ remains unclear when $n>2m-1$, we construct a candidate for $\im(\Psi_n)$ in this range with our third theorem.

\begin{theorem}\label{thm3}
If $n>2m-1$, then $\pi_n(\bbe_m)$ contains a subgroup isomorphic to \[\prod_{\bbn}\pi_n(S^m)\oplus \bigoplus_{2\leq j\leq \frac{n-1}{m-1}}\left(\prod_{\bbn}\bigoplus_{\bbn}\pi_{n}(S^{mj-j+1})\right)\]
on which the canonical homomorphism $\Psi_{n}:\pi_n(\bbe_m)\to \check{\pi}_{n}(\bbe_m)$ is injective.
\end{theorem}

The author is grateful to Kazuhiro Kawamura for helpful conversations regarding the content of this paper.

\section{Preliminaries}\label{sectionprelim}

All topological spaces in this paper are assumed to be Hausdorff. Throughout, $\ui$ will denote the closed unit interval, $S^n$ will denote the unit $n$-sphere $\{(x_0,\dots,x_n)\in\bbr^{n+1}\mid \sum x_{i}^{2}=1\}$ with basepoint $(1,0,0,\dots,0)$. We only consider higher homotopy groups in this paper and so we generally assume $n\geq 2$. For a based topological space $(X,x_0)$, we write $\Omega^{n}(X,x_0)$ to denote the space of relative maps $(I^n,\partial I^n)\to (X,x_0)$ and $\pi_n(X,x_0)=\{[f]\mid f\in \Omega^n(X,x_0)\}$ to denote the $n$-th homotopy group. When the basepoint is clear from context, we may simplify this notation to $\Omega^n(X)$ and $\pi_n(X)$. Fixing a based homeomorphism $I^n/\partial I^n\cong S^n$, we identify $\Omega^{n}(X)$ with the space of based maps $S^n\to X$. For $f\in \Omega^{n}(X)$, let $-f$ denote the reverse map given by $-f(t_1,t_2,\dots,t_n)=f(1-t_1,t_2,\dots,t_n)$. We write $f_1+f_2+\dots+f_m$ for the usual first-coordinate concatenation of maps $f_i\in \Omega^n(X)$. If $m$ is a positive integer, then $m f$ denotes the $m$-fold sum $f+f+\cdots +f$, $0f$ denotes the constant map $c\in \Omega^n(X)$, and if $m$ is a negative integer, then $m f$ denotes the $(-m)$-fold sum $(-f)+(-f)+\cdots+(-f)$.

\subsection{Whitehead Products}

We use commutator bracket notation to denote Whitehead products both on the level of maps and homotopy classes. For an ordered pair of natural numbers $(p,q)\in\bbn^2$, we view $S^p\vee S^q$ naturally as a subspace of $S^p\times S^q$. Let $a_{p,q}:S^{p+q-1}\to S^p\vee S^q$ denote an attaching map of the unique $p+q$-cell in $S^p\times S^q$. If $f:S^p\to X$ and $g:S^q\to X$ are based maps, the composition $[ f,g]=(f\vee g)\circ a_{p,q}:S^{p+q-1}\to X$ is the \textit{Whitehead product} of the pair of maps $(f,g)$. This operation on maps descends to a well-defined bracket operation $[-,-]:\pi_p(X)\times \pi_q(X)\to \pi_{p+q-1}(X)$ on homotopy classes, which we also refer to as the \textit{Whitehead product}.

\begin{remark}\label{relations}
Whitehead products satisfy the following standard relations \cite{whiteheadjhc}. Let $p,q,r\geq 2$. For $\alpha,\alpha'\in \pi_p(X)$, $\beta,\beta'\in \pi_q(X)$, and $\gamma\in\pi_r(X)$, we have:
\begin{itemize}
\item (identity) $[\alpha,0]=0$ and $[0,\beta]=0$,
\item (bilinearity) $[\alpha+\alpha',\beta]=[\alpha,\beta]+[\alpha',\beta]$ and $[\alpha,\beta+\beta']=[\alpha,\beta]+[\alpha,\beta']$,
\item (graded symmetry) $[\alpha,\beta]=(-1)^{pq}[\beta,\alpha]$,
\item (graded Jacobi identity) \[(-1)^{pr}[[\alpha,\beta],\gamma]+(-1)^{pq}[[\beta,\gamma],\alpha]+(-1)^{rq}[[\gamma,\alpha],\beta]=0.\]
\end{itemize}
\end{remark}
We have need to apply some of the standard relations infinitely many times in a single line of algebraic reasoning. To do this it is necessary to know that whenever representing maps are ``small" we can apply these relations using ``small" homotopies. The following is an observation made in \cite{BrazasWHProducts}. 

\begin{proposition}
Suppose $X$ is a based space, $p,q\geq 2$, $f,f'\in \Omega^p(X)$, and $g,g'\in \Omega^q(X)$.
\begin{enumerate}
\item If $H$ is a homotopy from $f$ to $f'$ and $G$ is a homotopy from $g$ to $g'$, then $[f,g]\simeq [f',g']$ by a homotopy with image equal to $\im(H)\cup \im(G)$,
\item If $f$ or $g$ is constant, then $[f,g]$ is null-homotopic in $\im(f)\cup\im(g)$,
\item We have $[f+f',g]\simeq [f,g]+[f',g]$ by a homotopy with image $\im(f)\cup\im(f')\cup \im(g)$ and $[f,g+g']\simeq [f,g]+[f,g']$ by a homotopy with image $\im(f)\cup\im(g)\cup \im(g')$.
\end{enumerate}
\end{proposition}

A based map $f:X\to Y$ induces a homomorphism $f_{\#}:\pi_n(X)\to \pi_n(Y)$ and $[f]$ will denote the based homotopy class of $f$. Recall that ordinary composition $f\circ g$ of maps descends to composition $[f]\circ [g]$ of homotopy classes in the homotopy category.

\begin{remark}\label{inducedmapremark}
If $h:X\to Y$, $f:S^p\to X$, and $g:S^q\to X$ are based maps, then we have equality $h\circ [f,g]=[h\circ f,h\circ g]$ on the level of continuous maps. For homotopy classes $a=[f]$ and $b=[g]$, we have $h_{\#}([a,b])=[h_{\#}(a),h_{\#}(b)]$.
\end{remark}

\subsection{Infinite sums}

\begin{definition}
Let $(X,x_0)$ be a based space. We say that a sequence $\{f_j\}_{j\in \bbn}$ of based maps $f_j:W_j\to X$ \textit{converges to} $x_0$ if for every neighborhood $U$ of $x_0$, we have $\im(f_j)\subseteq U$ for all but finitely many $j\in \bbn$. The \textit{infinite sum} of a sequence $\{f_j\}_{j\in \bbn}$ of based maps $f_j:S^n\to X$ that converges to $x_0$ is the map $\sum_{j=1}^{\infty}f_j\in \Omega^n(X,x_0)$ defined as $f_j$ on $\left[\frac{j-1}{j},\frac{j}{j+1}\right]\times I^{n-1}$ for all $j\in\bbn$ and which maps $\{1\}\times I^{n-1}$ to $x_0$.
\end{definition}

Infinite sums of maps $S^n\to X$ are infinitely commutative up to homotopy \cite{EK00higher} and thus satisfy a variety of identities. We refer to \cite{BrazasWHProducts} for the proofs of the next four Propositions.

\begin{proposition}\label{linearity}\cite{BrazasWHProducts}
If $\{f_j\}_{j\in \bbn}$, $\{g_j\}_{j\in \bbn}$ are sequences of maps $S^n\to X$ that converge to $x_0\in X$ and $a,b\in\bbz$, then $\sum_{j=1}^{\infty}af_j+bg_j\simeq a\sum_{j=1}^{\infty}f_j+b\sum_{j=1}^{\infty}g_j$ by a homotopy with image in $\bigcup_{j\in\bbn}\im(f_j)\cup\im(g_j)$.
\end{proposition}

\begin{definition}
Let $\{f_j\}_{j\in \bbn}$ and $\{g_j\}_{j\in \bbn}$ be sequences of maps $S^n\to X$ that converge to $x_0\in X$. We say that $\{f_j\}_{j\in \bbn}$ and $\{g_j\}_{j\in \bbn}$ are \textit{sequentially homotopic} and we write $\{f_j\}_{j\in \bbn}\simeq\{g_j\}_{j\in \bbn}$ if there exists a sequence $\{H_j\}_{j\in \bbn}$ where $H_j:S^n\times\ui\to X$ is a based homotopy from $f_j$ to $g_j$ (for each $j\in\bbn$) such that $\{H_j\}_{j\in\bbn}$ converges to $x_0$.
\end{definition}

\begin{proposition}\cite{BrazasWHProducts}\label{sumhomotopyprop}
If $\{H_j\}_{j\in\bbn}$ is a sequential homotopy between sequences $\{f_j\}_{j\in \bbn}$ and $\{g_j\}_{j\in \bbn}$ of maps $f_j,g_j\in \Omega^n(X)$ that converge to $x_0$, then $\sum_{j=1}^{\infty}f_j\simeq \sum_{j=1}^{\infty}g_j$ by a homotopy with image $\bigcup_{j\in\bbn}\im(H_j)$.
\end{proposition}

We also require the following relations for combinations of infinite sums and Whitehead products.

\begin{proposition}\label{wh1prop}\cite{BrazasWHProducts}
Let $(X,x_0)$ be a based space and $p,q\geq 2$. Suppose $\{f_j\}_{j\in \bbn}$, $\{g_j\}_{j\in \bbn}$ are sequences of maps $S^p\to X$ that converge to $x_0$ and $\{h_j\}_{j\in \bbn}$, $\{k_j\}_{j\in \bbn}$ are sequences of maps $S^q\to X$ that converge to $x_0$. Then $\{[f_j,h_j]\}_{j\in\bbn}$ and $\{[g_j,k_j]\}_{j\in\bbn}$ are sequences of maps $S^{p+q-1}\to X$ that converge to $x_0$. Moreover, if $\{f_j\}_{j\in \bbn}\simeq\{g_j\}_{j\in \bbn}$ and $\{h_j\}_{j\in \bbn}\simeq\{k_j\}_{j\in \bbn}$, then $\{[f_j,h_j]\}_{j\in\bbn}\simeq\{[g_j,k_j]\}_{j\in\bbn}$ and $\sum_{j=1}^{\infty}[f_j,h_j]\simeq \sum_{j=1}^{\infty}[g_j,k_j]$.
\end{proposition}

\begin{proposition}\label{wh2prop}\cite{BrazasWHProducts}
Let $(X,x_0)$ be a based space and $p,q\geq 2$. Suppose $\{f_j\}_{j\in \bbn}$, $\{g_j\}_{j\in \bbn}$ are sequences of maps $S^p\to X$ that converge to $x_0$ and $\{h_j\}_{j\in \bbn}$, $\{k_j\}_{j\in \bbn}$ are sequences of maps $S^q\to X$ that converge to $x_0$. Then
\begin{enumerate}
\item $\sum_{j=1}^{\infty}[f_j+g_j,h_j]\simeq \sum_{j=1}^{\infty}[f_j,h_j]+\sum_{j=1}^{\infty}[g_j,h_j]$ by a homotopy with image in $\bigcup_{j\in\bbn}\im(f_j)\cup\im(g_j)\cup\im(h_j)$,
\item $\sum_{j=1}^{\infty}[f_j,h_j+k_j]\simeq \sum_{j=1}^{\infty}[f_j,h_j]+\sum_{j=1}^{\infty}[f_j,k_j]$ by a homotopy with image in $\bigcup_{j\in\bbn}\im(f_j)\cup\im(h_j)\cup\im(k_j)$.
\end{enumerate}
\end{proposition}

\subsection{Shrinking wedges and their \v{C}ech homotopy groups}\label{sectionshrinkingwedge}

\begin{definition}
The \textit{shrinking wedge} of countable set $\{X_j\}_{j\in J}$ of based spaces is the space $\sw_{j\in J}X_j$ whose underlying set is the usual one-point union $\bigvee_{j\in J}X_j$ with canonical basepoint denoted $b_0$. A set $U$ is open in $\sw_{j\in J}X_j$ if
\begin{itemize}
\item $U\cap X_j$ is open in $X_j$ for all $j\in J$,
\item and whenever $b_0\in U$, we have $X_j\subseteq U$ for all but finitely many $j\in J$.
\end{itemize}
We refer to the special case $\bbe_m=\sw_{j\in\bbn}S^m$ as the \textit{$m$-dimensional earring space}.
\end{definition}

We note two other useful ways to construct a shrinking wedge $X=\sw_{j\in J}X_j$. First, $X$ canonically embeds in the direct product $\prod_{j\in J}X_j$. We let $\sigma:X\to \prod_{j\in J}X_j$ denote the embedding and often identify $X$ with its image under $\sigma$ so that $\sigma$ is an inclusion map. Second, is the construction of $\sw_{j\in J}X_j$ as an inverse limit. Let $X_{\leq k}=\bigvee_{j=1}^{k}X_j$ be the wedge of the first $k$-many spaces and $b_{k+1,k}:X_{\leq k+1}\to X_{\leq k}$ denote the map that collapses $X_{k+1}$ to the wedgepoint. The projection maps $b_k:X\to X_{\leq k}$ collapsing the sub-shrinking wedge $X_{\geq k+1}=\sw_{j>k}X_j$ to the basepoint induce a homeomorphism $X\to \varprojlim_{k}(X_{\leq k},b_{k+1,k})$.
%This construction indicates that a map $f:Z\to Y$ is continuous if and only if $b_k\circ f$ is continuous for all $k$. 

\begin{theorem}\label{exacttheorem}\cite{Kawamurasuspensions}
Let $X=\sw_{j\in\bbn}X_j$ be the shrinking wedge of a sequence of path-connected, based spaces $\{X_j\}_{j\in\bbn}$. For all $n\geq 2$, the long exact sequence of the pair $(\prod_{j\in\bbn}X_j,X)$ gives short exact sequences that split naturally:
\[\xymatrix{0 \ar[r] &  \pi_{n+1}(\prod_{j\in\bbn}X_j,X) \ar[r]^-{\partial} &  \pi_n(X) \ar[r]^-{\sigma_{\#}} & \pi_n(\prod_{j\in\bbn}X_j) \ar[r] & 0.}\]
Moreover, the splitting homomorphism $\zeta:\pi_n(\prod_{j\in\bbn}X_j)\to \pi_n(X)$ is given by the infinite sum operation $\zeta([(f_j)_{j\in\bbn}])=\left[\sum_{j=1}^{\infty}f_j\right]$ where $(f_j)_{j\in\bbn}$ denotes the map $S^n\to \prod_{j\in\bbn}X_j$, which is $f_j$ in the $j$-th coordinate. 
\end{theorem}
Theorem \ref{exacttheorem} implies that $\pi_n(\sw_{j\in\bbn}X_j)$ is isomorphic to: \[\pi_{n+1}\left(\prod_{j\in\bbn}X_j,X\right)\oplus\pi_n\left(\prod_{j\in\bbn}X_j\right)\cong \pi_{n+1}\left(\prod_{j\in\bbn}X_j,X\right)\oplus\prod_{j\in\bbn}\pi_n(X_j)\]
Hence, to characterize the group $\pi_n(X)$, $n\geq 2$ in terms of the homotopy groups of $X_j$, it suffices to characterize the relative homotopy group $\pi_{n+1}(\prod_{j\in\bbn}X_j,X)$. A common tactic for distinguishing elements of such groups comes from shape theory \cite{MS82}.

While the \v{C}ech homotopy groups (sometimes called the shape homotopy groups) of any based space are well-defined, we have a simple description in the following case: if $X=\sw_{j\in\bbn}X_j$ is a shrinking wedge of path-connected, compact, based polyhedra, then the inverse limit group $\check{\pi}_n(X)=\varprojlim_{k}(\pi_n(X_{\leq k}),b_{k+1,k\#})$ is the \textit{$n$-th \v{C}ech homotopy group} of $X$ (this presentation of $\check{\pi}_n(X)$ is justified by \cite[Ch I, \textsection 5.4, Theorem 13]{MS82}). The maps $b_k:X\to X_{\leq k}$ induce a canonical homomorphism $\Psi_n:\pi_n(X)\to \check{\pi}_n(X)$, $\Psi_n([f])=([b_k\circ f])_{k\in\bbn}$, which we call the \textit{$n$-th \v{C}ech homotopy map}.

\begin{lemma}\label{exactlemma}
For any sequence of path-connected, compact, based polyhedra $\{X_j\}_{j\in\bbn}$ and $n\geq 2$, there is an exact sequence
\[\xymatrix{0 \ar[r] &  \check{\pi}_{n+1}\left(\prod_{j\in\bbn}X_j,X\right) \ar[r]^-{\check{\partial}} &  \check{\pi}_n(X) \ar[r]^-{\check{\sigma}} & \pi_n(\prod_{j\in\bbn}X_j) \ar[r] & 0}\]
that splits naturally.
\end{lemma}

\begin{proof}
We utilize the notation above for the finite wedge $X_{\leq k}=\bigvee_{j=1}^{k}X_j$ and bonding maps $b_{k+1,k}$. Let $\sigma_{k}:X_{\leq k}\to \prod_{j=1}^{k}X_j$ be the canonical embedding and let $s_k:\pi_n(\prod_{j=1}^{k}X_j)\to \pi_n(X_{\leq k})$ be the canonical section to the induced homomorphism $\sigma_{k\#}:\pi_n(X_{\leq k})\to \pi_n(\prod_{j=1}^{k}X_j)$. We have an inverse sequence of short exact sequences with bonding maps appearing vertically in the following diagram:
\[\xymatrix{0 \ar[r] &  \pi_{n+1}(\prod_{j=1}^{k+1}X_j,X_{\leq k+1}) \ar[d]^-{p_{k+1,k\#}} \ar[r]^-{\partial_{k+1}} &  \pi_n(X_{\leq k+1})  \ar[d]^-{b_{k+1,k\#}}  \ar[r]^-{\sigma_{k+1\#}} & \pi_n(\prod_{j=1}^{k+1}X_j)  \ar[d]^-{p_{k+1,k\#}}  \ar[r] & 0\\
0 \ar[r] &  \pi_{n+1}(\prod_{j=1}^{k}X_j,X_{\leq k}) \ar[r]^-{\partial_{k}} &  \pi_n(X_{\leq k}) \ar[r]^-{\sigma_{k\#}} & \pi_n(\prod_{j=1}^{k}X_j) \ar[r] & 0
}\]
In the above diagram, $p_{k+1,k}$ is the projection map and $\partial_k$ is the usual boundary map. The vertical morphism of exact sequences admits a canonical section induced by the based inclusion maps $\prod_{j=1}^{k}X_j\to \prod_{j=1}^{k+1}X_j$. Hence, the inverse sequence satisfies the Mittag-Leffler condition and, in the limit, we obtain an exact sequence. Set $\check{\partial}=\varprojlim_{k}\partial$ and $\check{\sigma}=\varprojlim_{k}\sigma_{k\#}$. Since the space $X_j$ are compact polyhedra, the pairs $(\prod_{j=1}^{k}X_j,X_{\leq k})$, $k\in\bbn$ form an $\mathbf{HPol_{\ast}^{2}}$-expansion for the pair $\left(\prod_{j\in\bbn}X_j,X\right)$. Hence, we may identify $ \check{\pi}_{n+1}\left(\prod_{j\in\bbn}X_j,X\right)$ with $\varprojlim_{k}\pi_{n+1}\left(\prod_{j=1}^{k}X_j,X_{\leq k}\right)$ \cite[Ch I, \textsection 5.4, Theorem 13]{MS82}. Under this identification, the inverse limit of exact sequences is the desired sequence in the statement of the lemma. Since the canonical sections $s_{k}$ of the maps $\sigma_{k\#}$ are natural, $\varprojlim_{k}s_k$ is a section of $\check{\sigma}$.
\end{proof}

Lemma \ref{exactlemma} implies that $\check{\pi}_n\left(\sw_{j\in\bbn}X_j\right)$ splits as
\[\check{\pi}_n\left(\sw_{j\in\bbn}X_j\right)\cong \check{\pi}_{n+1}\left(\prod_{j\in\bbn}X_j,\sw_{j\in\bbn}X_j\right)\oplus \pi_n\left(\prod_{j\in\bbn}X_j\right).\]
We refer to the image of the splitting map $\pi_n\left(\prod_{j\in\bbn}X_j\right)\to \check{\pi}_n\left(\sw_{j\in\bbn}X_j\right)$ as the \textit{weight-$1$ summand} of $\check{\pi}_n\left(\sw_{j\in\bbn}X_j\right)$. This terminology will be justified later on. We also obtain the following morphism of exact sequences induced by the projection maps $b_k$.
\[\xymatrix{
0 \ar[r] &  \pi_{n+1}\left(\prod_{j\in\bbn}X_j,\sw_{j\in\bbn}X_j\right) \ar[d]_-{\Psi_n} \ar[r]^-{\partial} &  \pi_n(\sw_{j\in\bbn}X_j) \ar[d]^-{\Psi_n} \ar[r]^-{\sigma_{\#}} & \pi_n(\prod_{j\in\bbn}X_j) \ar@{=}[d] \ar[r] & 0\\
0 \ar[r] &  \check{\pi}_{n+1}\left(\prod_{j=1}^{k}X_j,X_{\leq k}\right) \ar[r]^-{\check{\partial}} &  \check{\pi}_n(\sw_{j\in\bbn}X_j) \ar[r]^-{\check{\sigma}} & \pi_n(\prod_{j\in\bbn}X_j) \ar[r] & 0}\]
If we identify $\pi_n(\sw_{j\in\bbn}X_j)$ and $\check{\pi}_n(\sw_{j\in\bbn}X_j)$ with their splittings, we see that $\Psi_n$, indicated by the middle arrow, becomes identified with the identity map of $\pi_n(\prod_{j\in\bbn}X_j)$.

\section{Applying the Hilton-Milnor Theorem}\label{hiltonsthmsection}

For the remainder of the paper, fix monotone sequence of integers $1\leq r_1\leq r_2\leq r_3\leq \cdots$ and set $X=\sw_{i\in\bbn}S^{r_i+1}$ with wedgepoint $x_0$. Later, we will restrict to the case where $r_i=m-1$ for all $i$ so that $X=\bbe_m$. As before, let $X_{\leq k}=\bigvee_{i=1}^{k}S^{r_i+1}$ be the finite wedge of the first $k$-spheres so that $X=\varprojlim_{k}(X_{\leq k},b_{k+1,k})$ with bonding maps $b_{k+1,k}:X_{\leq k+1}\to X_{\leq k}$ and projection maps and $b_k:X\to X_{\leq k}$. We also write $X_{\geq k}=\sw_{i=k}^{\infty}S^{r_i+1}$ for the sub-shrinking wedge that excludes the first $(k-1)$ wedge-summands. Our first aim is to describe the $n$-th \v{C}ech homotopy group $\check{\pi}_n(X)=\varprojlim_{k}(\pi_n(X_{\leq k}),b_{k+1,k\#})$.%The projection maps $b_k$ induce a canonical homomorphism $\Psi_{n}:\pi_n(X)\to \check{\pi}_n(X)$, $\Psi_n(f)=(b_{k\#}(f))_{k}$.

A non-trivial theorem of Eda-Kawamura (see \cite{EK00higher}) establishes some cases using higher connectivity. Since each $S^{r_i+1}$ is an $r_i$-connected CW-complex, we have $\pi_n(X)=\check{\pi}_n(X)=0$ for $1\leq n\leq r_1$. When $n=r_1+1$, we have that $\check{\pi}_n(X)\cong \prod_{i=1}^{\infty}\pi_n(S^{r_i+1})$ and that $\Psi_n$ is an isomorphism. Hence, \[\pi_{r_1+1}(X)\cong\check{\pi}_{r_1+1}(X)\cong \prod_{r_j=r_1}\pi_{r_1+1}(S^{r_j+1})\cong \prod_{r_j=r_1}\bbz,\]which is either finitely generated free abelian or the Baer-Specker group. Also it is shown in \cite{Kawamurasuspensions} that $\check{\pi}_{n+1}(\bbe_n)$ is isomorphic to $\bbz^{\bbn}$ when $n=2$ and $\bbz_{2}^{\bbn}$ when $n\geq 3$. 

To extend the argument used in \cite{Kawamurasuspensions}, we recall the notion of a Hall set, which has import in both group theory and Lie theory. In particular, a Hall set on a non-empty finite set $A_k=\{a_1,a_2,\dots,a_k\}$ provides a convenient basis for a free Lie algebra on $A_k$ \cite{HallMarsh}. Following Hilton's approach in \cite{HiltonSpheres}, we consider such a Hall set on $A_k=\{a_1,a_2,\dots,a_k\}$ where $a_i=[id_{S^{r_i+1}}]$ is the canonical generator of $\pi_{r_i+1}(S^{r_i+1})$.

\begin{definition}
A \textit{Hall set} on a non-empty finite set $A_k=\{a_1,a_2,\dots,a_k\}$ is a linearly ordered set $H_k$ of nested formal commutators determined by the following inductive collection and ordering process \cite{HallPhilip}: the set $H_{k}(1)$ of Hall words of \textit{weight $1$} is the set $A_k$ ordered as $a_1<a_2<\dots<a_k$. Suppose the set $H_k(i)$ of Hall words of weight $i$ has been defined for $i<m$ and that $\bigcup_{i<m}H_k(i)$ has been linearly ordered and so that if $i<j$, $x\in H_k(i)$, and $y\in H_k(j)$, then $x<y$. With this inductive hypothesis, we define a Hall word of weight $m$ to be a formal commutator $[x,y]$ such that
\begin{enumerate}
\item $x\in H_k(i)$ and $y\in H_k(j)$ where $i+j=m$,
\item $x<y$ in $\bigcup_{i<m}H_k(i)$,
\item if $y=[a,b]$ for $a,b\in \bigcup_{i<m}H_k(i)$, then $a\leq x$.
\end{enumerate}
We let $H_k(m)$ be the set of all Hall words of weight $m$ with an arbitrary linear ordering. Now $H_k=\bigcup_{i=1}^{\infty}H_k(i)$ is given a linear ordering that extends the linear ordering of each $\bigcup_{i<m}H_k(i)$ so that a Hall word of smaller weight is always less than another Hall word of greater weight.
\end{definition}

For each $i\in\{1,2,\dots,k\}$ and Hall word $w\in H_k$, we let $\nu_i(w)$ denote the number of times the letter $a_i$ appears in $w$ as a word (when the brackets are removed). The weight of a Hall word $w\in H_k$ is the same as the length $L(w)=\sum_{i=1}^{k}\nu_i(w)$ of $w$ as a word. 

\begin{example}
When $k=2$, the first elements of the Hall set $H_2$ are ordered as follows (where an arbitrary ordering is chosen for words at a fixed weight): $a_1$, $a_2$, $[a_1,a_2]$, $[a_1,[a_1,a_2]]$, $[a_2,[a_1,a_2]]$, $[a_1,[a_1,[a_1,a_2]]]$, $[a_2,[a_1,[a_1,a_2]]]$, $[a_2,[a_2,[a_1,a_2]]]$, $\dots$.
%
%When $k=3$, the first elements of the Hall set $H_3$ are ordered as follows (where an arbitrary ordering is chosen for words at a fixed weight):\\
%
%\noindent $a_1$, $a_2$, $a_3$, $[a_1,a_2]$, $[a_1,a_3]$, $[a_2,a_3]$, $[a_1,[a_1,a_2]]$, $[a_1,[a_1,a_3]]$, $[a_2,[a_1,a_2]]$, $[a_2,[a_1,a_3]]$, $[a_2,[a_2,a_3]]$, $[a_3,[a_1,a_2]]$, $[a_3,[a_1,a_3]]$, $[a_3,[a_2,a_3]]$, $[a_1,[a_1,[a_1,a_2]]]$, $[a_1,[a_1,[a_1,a_3]]]$, $[a_2,[a_1,[a_1,a_2]]]$, $[a_2,[a_1,[a_1,a_3]]]$, $[a_2,[a_2,[a_1,a_2]]]$, $[a_2,[a_2,[a_1,a_3]]]$, $[a_2,[a_2,[a_2,a_3]]]$, $[a_3,[a_1,[a_1,a_2]]]$, $[a_3,[a_1,[a_1,a_3]]]$, $[a_3,[a_2,[a_1,a_2]]]$, $[a_3,[a_2,[a_1,a_3]]]$, $[a_3,[a_2,[a_2,a_3]]]$, $[a_3,[a_3,[a_1,a_2]]]$, $[a_3,[a_3,[a_1,a_3]]]$, $[a_3,[a_3,[a_2,a_3]]]$, $[[a_1,a_2],[a_1,a_3]]$, $[[a_1,a_2],[a_2,a_3]]$, $[[a_1,a_3],[a_2,a_3]]$, $\dots$
\end{example}

The number of Hall words $w\in H_{k}$ of weight $j$ is the value of necklace polynomial $M_{k}(j)=\frac{1}{j}\sum_{d\mid j}\mu(d)k^{j/d}$ where $\mu$ is the M\"{o}bius function. Next, we show that since $A_2\subseteq A_3\subseteq A_4\subseteq \cdots$, we can choose Hall sets $H_k$ on $A_k$ in a coherent way.

\begin{proposition}\label{nestedhallset}
There exists a sequence of linearly ordered sets $H_2\subseteq H_3\subseteq H_4\subseteq \cdots$ where $H_k$ is a Hall set on $A_k$ and $H_{k+1}\backslash H_{k}$ consists of Hall words involving the letter $a_{k+1}$.
\end{proposition}

\begin{proof}
Since $H_k(1)=A_k$, the given ordering on the sets $A_k$ ensures that $H_k(1)$ is an initial suborder of $H_{k+1}(1)$ and that $H_{k+1}(1)\backslash H_{k}(1)=\{a_{k+1}\}$. Suppose that, for all $i<m$, we have constructed the sets $H_k(i)$, $k\in\bbn$ so that $H_{k}(i)$ is an initial suborder of $H_{k+1}(i)$ and that $H_{k+1}(i)\backslash H_{k}(i)$ consists of all elements of $H_{k+1}(i)$ involving the letter $a_{k+1}$. This induction hypothesis and the Hall set construction described above ensure that $H_2(m)\subseteq H_3(m)\subseteq H_4(m)\subseteq \cdots$ as subsets. It follows that $H_{k+1}(m)\backslash H_{k}(m)$ is the set of elements of $H_{k+1}(m)$ that involve the letter $a_{k+1}$. We choose the ordering of $H_2(m)$ arbitrarily and recursively choose the ordering on $H_{k+1}(m)$, $k\geq 2$ so that $H_{k}(m)$ is an initial suborder of $H_{k+1}(m)$. This completes the induction and we obtain Hall sets $H_k$ on $A_k$ such that for all $k,i\in\bbn$, $H_{k}(i)\subseteq H_{k+1}(i)$ where $H_{k+1}(i)\backslash H_{k}(i)$ is the set of Hall words in $H_{k+1}$ of weight $i$ that involve the letter $a_{k+1}$. Thus $H_k\subseteq H_{k+1}$ for all $k\in\bbn$ and if $w\in H_{k+1}\backslash H_k$ has weight $i$, then $w\in H_{k+1}(i)\backslash H_{k}(i)$.
\end{proof}

Since we have chosen $a_i$ to be an element of $\pi_{r_i+1}(S^{r_i+1})$, a Hall word $w\in H_k$ may be identified with the corresponding iterated Whitehead product, which lies in some homotopy group of $X_{\leq k}$. To identify the dimension of this homotopy group, we borrow the following definition from \cite{HiltonSpheres}.

\begin{definition}
The \textit{height} of a Hall word $w\in H_k$ is $h(w)=\sum_{i=1}^{k}r_i\nu_i(w)$.
\end{definition}

\begin{remark}
For a Hall word $w\in H_k$, the height $h(w)$ is one less than the dimension of the homotopy group in which $w$ is an element, that is $w\in \pi_{h(w)+1}(X_{\leq k})$. For example when $k=3$, we have $a_1\in \pi_{r_1+1}(S^{r_1+1})$, $a_1\in \pi_{r_2+1}(S^{r_2+1})$, and $a_3\in \pi_{r_3+1}(S^{r_3+1})$. Then
\begin{enumerate}
\item $w=[a_1,a_2]\in \pi_{r_1+r_2+1}(X_3)$ has height $h(w)=r_1+r_2$, 
\item $w=[a_1,[a_1,a_2]]\in \pi_{2r_1+r_2+1}(X_3)$ has height $h(w)=2r_1+r_2$, 
\item $w=[a_1,[a_2,a_3]]\in \pi_{r_1+r_2+r_3+1}(X_3)$ has height $h(w)=r_1+r_2+r_3$.
\end{enumerate}
\end{remark}

Because a Hall word is given by nesting Whitehead products in a binary fashion, Relation (1) of Remark \ref{relations} makes the following proposition evident.

\begin{proposition}\label{trivializedprop}
Let $k\geq 1$ and $w\in H_k$ be a Hall word regarded as a nested Whitehead product in $\pi_{h(w)+1}(X_{\leq k})$. If we replace any letter of $w$ with $0$, then the resulting element of $\pi_{h(w)+1}(X_{\leq k})$ is trivial.
\end{proposition}

Using Proposition \ref{nestedhallset}, we fix Hall sets $H_k$ on $A_k$ so that $H_1\subseteq H_2\subseteq H_3\subseteq \cdots$. The following is Hilton's initial version of the well-known Hilton-Milnor Theorem. We choose to use Hilton's original formulation in terms of homotopy groups of spheres and Whitehead products since our focus is on shrinking wedges of spheres.

\begin{theorem}[Hilton \cite{HiltonSpheres}]\label{hiltonstheorem}
For given $k\geq 1$ and $n>r_1$, there is an isomorphism $\phi_{n,k}:\bigoplus_{w\in H_k}\pi_n(S^{h(w)+1})\to \pi_n(X_{\leq k})$ determined by $\phi_{n,k}(f)=w\circ f$ for each $f\in \pi_n(S^{h(w)+1})$.
\end{theorem}

\begin{remark}
Although the direct sum $\bigoplus_{w\in H_k}\pi_n(S^{h(w)+1})$ in Hilton's Theorem is indexed by the infinite set $H_k$, only finitely many summands are non-trivial since $n$ is fixed. Indeed, the set $H_{n,k}=\{w\in H_k\mid h(w)+1\leq n\}$ is finite and the domain of $\phi_{n,k}$ can be more efficiently expressed as the finite direct sum $\bigoplus_{w\in H_{n,k}}\pi_n(S^{h(w)+1})$.
\end{remark}

\begin{lemma}\label{commuteslemma}
Fix $k\geq 1$, $n>r_1$, and $w\in H_{n,k+1}$. If $w\in H_{n,k}$, then the homomorphism \[b_{k+1,k\#}\circ \phi_{n,k+1}:\bigoplus_{v\in H_{n,k+1}}\pi_n(S^{h(v)+1})\to \pi_{n}(X_{\leq k})\] maps the summand $\pi_n(S^{h(w)+1})$ isomorphically onto $\phi_{n,k}(\pi_n(S^{h(w)+1}))$ and if $w\in H_{n,k+1}\backslash H_{n,k}$, then $b_{k+1,k\#}\circ \phi_{n,k+1}$ trivializes $\pi_n(S^{h(w)+1})$.
\end{lemma}

\begin{proof}
If $f\in \pi_n(S^{h(w)+1})$, then $b_{k+1,k\#}\circ \phi_{n,k+1}(f)=[b_{k+1,k}]\circ w\circ f$. Note that $[b_{k+1,k}]\circ w$ is still a Whitehead product with the same nested structure but where each appearance of the letter $a_{k+1}$ is trivialized. Hence, if $a_{k+1}$ appears in $w$, then $[b_{k+1,k}]\circ w=0$ by Proposition \ref{trivializedprop}. It follows that if $w\in H_{n,k+1}\backslash H_{n,k}$ then $[b_{k+1,k}]\circ w\circ f=0$. On the other hand, if $w\in H_{n,k}$, then $a_{k+1}$ does not appear as a letter in $w$, which implies that $[b_{k+1,k}]\circ w=w$ and thus $b_{k+1,k\#}\circ \phi_{n,k+1}(f)=w\circ f=\phi_{n,k}(f)$.
\end{proof}

Define $H_{\infty}=\bigcup_{k=1}^{\infty}H_k$ and $H_{n,\infty}=\bigcup_{k=1}^{\infty}H_{n,k}=\{w\in H_{\infty} \mid h(w)+1\leq n\}$ and note that both of these sets are infinite.

\begin{theorem}\label{cechtheorem1}
If $n>r_1$, then there is a canonical isomorphism \[\phi_{n,\infty}:\prod_{w\in H_{n,\infty}}\pi_n(S^{h(w)+1})\to \check{\pi}_n(X)\] given by $\phi_{n,\infty}((f_w)_{w\in H_{n,\infty}})=\left(\sum_{w\in H_{n,k}}w\circ f_w\right)_{k\in\bbn}$.
\end{theorem}

\begin{proof}
For each $k\geq 1$, let $p_{k+1,k}:\bigoplus_{w\in H_{n,k+1}}\pi_n(S^{h(w)+1})\to\bigoplus_{w\in H_{n,k}}\pi_n(S^{h(w)+1})$ be the unique homomorphism, which trivializes $\pi_n(S^{h(w)+1})$ if $w\in H_{n,k+1}\backslash H_{n,k}$ and is the identity on $\pi_n(S^{h(w)+1})$ if $w\in H_{n,k}$. By Lemma \ref{commuteslemma}, the following diagram commutes where the horizontal maps are isomorphisms from Hilton's Theorem.
\[\xymatrix{
\bigoplus_{w\in H_{n,k+1}}\pi_n(S^{h(w)+1}) \ar[rr]^-{\phi_{n,k+1}} \ar[d]_-{p_{k+1,k}} && \pi_n(X_{\leq k+1}) \ar[d]^-{b_{k+1,k\#}} \\
\bigoplus_{w\in H_{n,k}}\pi_n(S^{h(w)+1}) \ar[rr]_-{\phi_{n,k}} && \pi_n(X_{\leq k})
}\]
Because the maps $p_{k+1,k}$ are projection maps of a finite direct product, the inverse limit of the left vertical morphisms is the direct product $\prod_{w\in H_{n,\infty}}\pi_n(S^{h(w)+1})$. Hence, taking the corresponding inverse limit of the isomorphisms $\phi_{n,k}$ as $k\to \infty$ yields the desired isomorphism.
\end{proof}

\begin{remark}
Since $\pi_n(S^{h(w)+1})=0$ whenever $w\in H_{\infty}\backslash H_{n,\infty}$, we can also express the isomorphism $\phi_{n,\infty}$ as a map
\[\phi_{n,\infty}:\prod_{w\in H_{\infty}}\pi_n(S^{h(w)+1})\to \check{\pi}_n(X)\] with the same formula as that given in Theorem \ref{cechtheorem1}.
\end{remark}

\begin{remark}
When $w=a_i\in H_{\infty}$ has weight $L(w)=1$, we have $h(w)=r_i$ and thus $w$ corresponds to a factor $\pi_n(S^{r_i+1})$ in the direct product of Theorem \ref{cechtheorem1}. Collecting all such factors together into a single summand, we have $\check{\pi}_n(X)\cong \prod_{L(w)>1}\pi_n(S^{h(w)+1})\oplus \prod_{i=1}^{\infty}\pi_n(S^{r_i+1})$ where $\prod_{i=1}^{\infty}\pi_n(S^{r_i+1})$ is the previously-defined weight-$1$ summand of $\check{\pi}_n(X)$. Hence, \[\check{\pi}_{n+1}\left(\prod_{i=1}^{k}S^{r_i+1},\bigvee_{i=1}^{k}S^{r_i+1}\right)\cong \prod_{L(w)>1}\pi_n(S^{h(w)+1}).\]
\end{remark}

We are particularly interested in the case where $r_i=m-1\geq 1$ so that $\sw_{i}S^{r_i+1}=\bbe_{m}$. To analyze this case more carefully, we partition $H_{\infty}$ as follows: Let $H_{\infty}(j)=\bigcup_{k=1}^{\infty}H_k(j)$ be the Hall words of weight $j$, e.g. $H_{\infty}(1)=\{a_1,a_2,a_3,\cdots\}$ and $H_{\infty}(2)=\{[a_i,a_j]\mid i<j\}$. Note that for each $j\geq 1$, $H_{\infty}(j)$ is countably infinite.

%\begin{theorem}\label{cechtheorem2}
%For all $n\geq m\geq 2$, we have \[\check{\pi}_n(\bbe_m)\cong \bigoplus_{1\leq j\leq \frac{n-1}{m-1}}\left(\pi_{n}(S^{mj-j+1})\right)^{\bbn}.\]
%\end{theorem}
Next, we prove our first main result.

\begin{proof}[Proof of Theorem \ref{thm1}]
By Theorem \ref{cechtheorem1}, $\check{\pi}_n(\bbe_m)$ is isomorphic to the direct product $\prod_{w\in H_{\infty}}\pi_n(S^{h(w)+1})$. We prove the current theorem by rearranging the terms in this product. In this case, we have $r_i=m-1$ for all $i$. Thus each Hall word $w\in H_k$ has height $h(w)=\sum_{i=1}^{k}(m-1)\nu_i(k)=(m-1)L(w)$ where $L(w)$ is the weight/length of $w$. Recall that we have defined $H_{\infty}(j)=\{w\in H_{\infty}\mid L(w)=j\}$ to be the Hall words of weight $j$. Thus
\[
\check{\pi}_n(\bbe_m) \cong \prod_{w\in H_{\infty}}\pi_n(S^{(m-1)L(w)+1})
\cong  \prod_{j=1}^{\infty}\prod_{w\in H_{\infty}(j)}\pi_n(S^{(m-1)j+1}).
\]
As $j\to\infty$, the groups $\pi_n(S^{(m-1)j+1})$ have the potential to be non-trivial only when $n\geq (m-1)j+1$ or, equivalently, when $j\leq \frac{n-1}{m-1}$. Hence, $\prod_{w\in H_{\infty}(j)}\pi_n(S^{(m-1)j+1})=0$ when $j> \frac{n-1}{m-1}$. We have:
\begin{eqnarray*}
 \prod_{j=1}^{\infty}\prod_{w\in H_{\infty}(j)}\pi_n(S^{(m-1)j+1}) &\cong & \bigoplus_{1\leq j\leq \frac{n-1}{m-1}}\prod_{w\in H_{\infty}(j)}\pi_n(S^{(m-1)j+1}) \\
&\cong & \bigoplus_{1\leq j\leq \frac{n-1}{m-1}}\left(\pi_n(S^{(m-1)j+1})\right)^{\bbn} \\
\end{eqnarray*}
where the last isomorphism holds since $m$ is fixed and $H_{\infty}(j)$ is countably infinite.
\end{proof}

We will often express $\check{\pi}_n(\bbe_m)$ as the direct sum $\bigoplus_{1\leq j\leq \frac{n-1}{m-1}}\mcw_j$ where $\mcw_j$ is the image of $\prod_{w\in H_{\infty}(j)}\pi_n(S^{(m-1)j+1})$ under the isomorphism $\phi_{n,\infty}$. We refer to $\mcw_j$ as the \textit{weight-$j$ summand} of $\check{\pi}_n(\bbe_m)$. Since $\mcw_1\cong\prod_{w\in H_{\infty}(1)}\pi_n(S^m)$ corresponds precisely to the ``weight-$1$ summand" of $\check{\pi}_n(\bbe_m)$ as defined in Section \ref{sectionshrinkingwedge}, we have the following.

\begin{corollary}
If $n> m\geq 2$, then \[\check{\pi}_{n+1}\left(\prod_{j\in\bbn}S^m,\bbe_m\right)\cong \bigoplus_{2\leq j\leq \frac{n-1}{m-1}}\left(\pi_{n}(S^{mj-j+1})\right)^{\bbn}.\]
\end{corollary}

\begin{example}\label{specialcaseexample}
Some special cases worth noting are the following.
\begin{enumerate}
\item When $n=3$ and $m=2$, we have a summand for both $j=1,2\leq \frac{3-1}{2-1}$. In this case,
\[\check{\pi}_3(\bbe_2)\cong (\pi_3(S^2))^{\bbn}\oplus (\pi_3(S^3))^{\bbn}\cong \bbz^{\bbn}\oplus\bbz^{\bbn}\]
Although this is abstractly isomorphic to $\bbz^{\bbn}$, this decomposition shows that the first summand of $\bbz^{\bbn}\oplus\bbz^{\bbn}$ is the weight-$1$ summand of $\check{\pi}_3(\bbe_2)$.
\item For the groups $\check{\pi}_{n+1}(\bbe_n)$, $n\geq 3$, we only have a summand for $j=1\leq \frac{n}{n-1}$. In this case, we enter the stable range and have
\[\check{\pi}_{n+1}(\bbe_n)\cong (\pi_{n+1}(S^n))^{\bbn}\cong \bbz_{2}^{\bbn},\]
the entirety of which is the weight-$1$ summand. This agrees with the computation in \cite{Kawamurasuspensions}.
\item For $n=4$ and $m=2$, we have a summand for $j=1,2,3\leq \frac{4-1}{2-1}$. In this case, we have
\[\check{\pi}_4(\bbe_2)\cong(\pi_4(S^2))^{\bbn}\oplus (\pi_4(S^3))^{\bbn}\oplus (\pi_4(S^4))^{\bbn}\cong \bbz_{2}^{\bbn}\oplus \bbz_{2}^{\bbn}\oplus \bbz^{\bbn}\]
where $\mcw_1\cong(\pi_4(S^2))^{\bbn}\cong\bbz_{2}^{\bbn}$ is the weight-$1$ summand.
More generally, when $n\geq 4$, we have
\[\check{\pi}_n(\bbe_2)\cong \pi_n(S^2)^{\bbn}\oplus \pi_n(S^3)^{\bbn}\oplus \cdots \oplus \pi_n(S^n)^{\bbn}\]where $\mcw_j=\pi_n(S^{j+1})^{\bbn}$ corresponds to the weight-$j$ summand.
\end{enumerate}
\end{example}

\begin{corollary}
For fixed $n\geq 0$, the \v{C}ech homotopy groups $\check{\pi}_{m+n}(\bbe_m)$ stabilize when $m\geq n+2$.
\end{corollary}

\begin{proof}
When $m\geq n+2$, we have $\frac{m+n-1}{m-1}<2$ and thus Theorem \ref{thm1} only gives a single summand corresponding to $j=1$. In particular, $\check{\pi}_{m+n}(\bbe_m)\cong \left(\pi_{m+n}(S^m)\right)^{\bbn}$. Since the groups $\pi_{m+n}(S^m)$ stabilize when $m\geq n+2$, the result follows.
\end{proof}

%\begin{problem}
%Do the groups $\pi_{m+n}(\bbe_n)$ stabilize when $n\geq m+2$?
%\end{problem}

\section{More on the image of $\Psi_n:\pi_n(\bbe_m)\to\check{\pi}_n(\bbe_m)$}\label{sectionimage}

To identify new elements of the image of the canonical homomorphism $\Psi_n:\pi_n(\bbe_m)\to\check{\pi}_n(\bbe_m)$, we use the isomorphism $\phi_{n,\infty}:\bigoplus_{1\leq j\leq \frac{n-1}{m-1}}\left(\pi_{n}(S^{mj-j+1})\right)^{\bbn}\to \check{\pi}_n(\bbe_m)$ as a book-keeping device. 
%Recall that $\mcw_{j}=\phi_{n,\infty}(\left(\pi_{n}(S^{mj-j+1})\right)^{\bbn})$ is the image of the $j$-th summand under $\phi_{n,\infty}$. Thus, to determine the image of $\Psi_n$ it suffices to determine how this image relates to $\mcw_j$ for each $j$.

\begin{remark}\label{genericremark}
The set $H_k(j)$ of all Hall words of weight $j$ involving letters $a_1,\dots,a_k$ provides a convenient indexing set for the elements of the weight-$j$ summand $\mcw_{j}=\phi_{n,\infty}(\left(\pi_{n}(S^{mj-j+1})\right)^{\bbn})$. According to the formula for $\phi_{n,\infty}$ in Theorem \ref{cechtheorem1}, a generic element of $\mcw_j$ has the form $\left(\sum_{w\in H_{k}(j)}w\circ f_{w}\right)_{k\in\bbn}$ where $f_w\in\pi_{n}(S^{mj-j+1})$ and $f_w$ may only be non-zero for finitely many $w\in H_{n,k}\cap H_k(j)$.
\end{remark}

\begin{proposition}\label{w1prop}
For any $m,n\geq 2$, $\mcw_1$ is a direct summand of $\im(\Psi_{n})$.
\end{proposition}

\begin{proof}
Recall from the end of Section 2, that we have $\check{\sigma}\circ \Psi_n=\sigma_{\#}$ where $\im(\sigma_{\#})=\phi_{n,\infty}^{-1}(\mcw_1)\cong\mcw_1$ and $\sigma_{\#}$ splits by Theorem \ref{exacttheorem}.
\end{proof}

\begin{corollary}\label{firstcasecor}
If $n<2m-1$, then $\ds \Psi_{n}:\pi_n(\bbe_m)\to \check{\pi}_n(\bbe_m)$ is a split epimorphism.
\end{corollary}

\begin{proof}
Since $\frac{n-1}{m-1}<2$, Theorem \ref{thm1} implies that $\check{\pi}_n(\bbe_m)=\mcw_1$. The splitting of $\Psi_n$ follows from Proposition \ref{w1prop}.
\end{proof}

Next, we consider the special case $n=2m-1$ at the edge of the unstable range. In this situation, we will need to apply infinite sums and Whitehead products described in Section \ref{sectionprelim}.

\begin{lemma}\label{edgelemma}
For $m\geq 2$, the canonical homomorphism $\ds \Psi_{2m-1}:\pi_{2m-1}(\bbe_m)\to \check{\pi}_{2m-1}(\bbe_m)$ is a split epimorphism.
\end{lemma}

\begin{proof}
First, we apply Theorem \ref{thm1} in the case $n=2m-1$. Since $\frac{2m-2}{m-1}<3$, we have $\check{\pi}_{2m-1}(\bbe_m)= \mcw_1\oplus \mcw_2$. A splitting homomorphism $\mcw_1\to \pi_{n}(\bbe_m)$ of $\Psi_{n}$ is established in Proposition \ref{w1prop} and so it suffices to produce a splitting homomorphism $\mcw_2\to \pi_{n}(\bbe_m)$. Recall from Remark \ref{genericremark}, that a generic element of $\mcw_2$ has the form $\alpha=\left(\sum_{w\in H_{k}(2)}w\circ f_{w}\right)_{k\in\bbn}$ where $f_w\in \pi_{2m-1}(S^{2m-1})$. Since $H_{k}(2)=\{[a_i,a_j]\mid 1\leq i<j\leq k\}$, write $f_{i,j}$ for $f_w$ when $w=[a_i,a_j]$. Let $\gamma=[id_{S^{2m-1}}]$ be the generator of $\pi_{2m-1}(S^{2m-1})$ and write $f_{i,j}=\epsilon_{i,j}\gamma$ for integer $\epsilon_{i,j}\in \bbz$. Since $w\circ -$ is a homomorphism, we have $w\circ f_{i,j}=[a_i,a_j]\circ \epsilon_{i,j}\gamma=\epsilon_{i,j}([a_i,a_j]\circ\gamma)=\epsilon_{i,j}[a_i,a_j]=[a_i,\epsilon_{i,j}a_j]$. Altogether, we can express $\alpha$ more conveniently as
\[\alpha=\left(\sum_{i=1}^{k-1}\sum_{j=i+1}^{k}[a_i,\epsilon_{i,j}a_j]\right)_{k\in\bbn}\]
where the sums in the sequence are directly extended as $k$ increases.

Let $\ell_j:S^m\to \bbe_m$ be the inclusion of the $j$-th sphere so that $a_j=[\ell_j]$ and define an infinite sum map $F_{\alpha}:S^{2m-1}\to \bbe_m$ by\[F_{\alpha}=\sum_{i=1}^{\infty}\left[\ell_i,\sum_{j=i+1}^{\infty}\epsilon_{i,j}\ell_j\right]\]
This sum is continuous since the sequence $\left\{\left[\ell_i,\sum_{j=i+1}^{\infty}\epsilon_{i,j}\ell_j\right]\right\}_{i\in\bbn}$ converges to the wedgepoint of $\bbe_m$. Fix $k\in\bbn$ and let $b_{k}:\bbe_{m}\to X_{\leq k}$ denote the projection map to the union of the first $k$-spheres. For given $i\in\bbn$, note that $b_k\circ \left[\ell_i,\sum_{j=i+1}^{\infty}\epsilon_{i,j}\ell_j\right]$ is constant if $i>k$. If $i\leq k$, then
\begin{eqnarray*}
b_k\circ \left[\ell_i,\sum_{j=i+1}^{\infty}\epsilon_{i,j}\ell_j\right] &= & \left[b_k\circ\ell_i,b_k\circ\sum_{j=i+1}^{\infty}\epsilon_{i,j}\ell_j\right]\\
&=& \left[b_k\circ\ell_i,\sum_{j=i+1}^{\infty}b_k\circ\epsilon_{i,j}\ell_j\right].
\end{eqnarray*}
and we consider two cases. If $i=k$, this is equal to $[\ell_i,c]$ where $c$ is constant. Such a map is null-homotopic within the $i$-th wedge-summand of $\bbe_m$. Lastly, if $i<k$, this is homotopic to $\left[\ell_i,\sum_{j=i+1}^{k}\epsilon_{i,j}\ell_j\right]$. All three of these cases are realized by homotopies having image in $X_{\geq i}$. Thus, Proposition \ref{sumhomotopyprop} gives 
\[b_k\circ F_{\alpha}\simeq\sum_{i=1}^{k-1}\left[\ell_i,\sum_{j=i+1}^{k}\epsilon_{i,j}\ell_j\right].\]
On homotopy classes, we have\[b_{k\#}([F_{\alpha}])=\sum_{i=1}^{k-1}\left[a_i,\sum_{j=i+1}^{k}\epsilon_{i,j}a_j\right]=
\sum_{i=1}^{k-1}\sum_{j=i+1}^{k}\left[a_i,\epsilon_{i,j}a_j\right]\]
where the second equality follows from bilinearity of Whitehead products. Since $k$ was arbitrary, we conclude that $\Psi_{n}([F_{\alpha}])=(b_{k\#}([F_{\alpha}]))_{k\in\bbn}=\alpha$.

For the splitting, it suffices to check that the mapping $\alpha\mapsto [F_{\alpha}]$ defines a homomorphism $\mcw_{2}\to \pi_{2m-1}(\bbe_m)$. Since our representation of $\alpha$ is uniquely determined by the integers $\epsilon_{i,j}$, $j>i$, this function is well-defined. We show that $[F_{\alpha+\beta}]=[F_{\alpha}]+[F_{\beta}]$. Express $\beta\in\mcw_2$ as $\beta=\left(\sum_{i=1}^{k-1}\sum_{j=i+1}^{k}[a_i,\delta_{i,j}a_j]\right)_{k\in\bbn}$ for integers $\delta_{i,j}$, $j>i$ and note that $\alpha+\beta=\left(\sum_{i=1}^{k-1}\sum_{j=i+1}^{k}[a_i,(\epsilon_{i,j}+\delta_{i,j})a_j]\right)_{k\in\bbn}$.

Recalling the formula for $F_{\alpha}$, $F_{\beta}$, and $F_{\alpha+\beta}$, and applying the infinite-sum homotopies from Section \ref{sectionprelim}, we have the following sequence:
\begin{eqnarray*}
F_{\alpha}+F_{\beta} &= & \sum_{i=1}^{\infty}\left[\ell_i,\sum_{j=i+1}^{\infty}\epsilon_{i,j}\ell_j\right]+\sum_{i=1}^{\infty}\left[\ell_i,\sum_{j=i+1}^{\infty}\delta_{i,j}\ell_j\right]\\
&\simeq & \sum_{i=1}^{\infty}\left(\left[\ell_i,\sum_{j=i+1}^{\infty}\epsilon_{i,j}\ell_j\right]+ \left[\ell_i,\sum_{j=i+1}^{\infty}\delta_{i,j}\ell_j\right]\right)\\
&\simeq & \sum_{i=1}^{\infty}\left[\ell_i,\sum_{j=i+1}^{\infty}\epsilon_{i,j}\ell_j+\sum_{j=i+1}^{\infty}\delta_{i,j}\ell_j\right]\\
&\simeq & \sum_{i=1}^{\infty}\left[\ell_i,\sum_{j=i+1}^{\infty}\epsilon_{i,j}\ell_j+\delta_{i,j}\ell_j\right]\\
&\simeq & \sum_{i=1}^{\infty}\left[\ell_i,\sum_{j=i+1}^{\infty}(\epsilon_{i,j}+\delta_{i,j})\ell_j\right].\\
&= & F_{\alpha+\beta}
\end{eqnarray*}
\end{proof}

Corollary \ref{firstcasecor} and Lemma \ref{edgelemma} complete the proof of Theorem \ref{thm2}.

In the unstable range $n>2m-1$, the same argument for surjectivity does not work and it does not appear that Hilton-Milnor invariants overcome the obstacle. Fix $0\neq \beta\in \pi_n(S^{2m-1})$. Note that the element\[\left(\left[a_1,\sum_{i=2}^{k}a_i\right]\circ\beta\right)_{k\in\bbn} =\left(\left(\sum_{i=2}^{k}[a_1,a_i]\right)\circ\beta\right)_{k\in\bbn} \]of $\mcw_2$ is the image of $\left[\ell_1,\sum_{j=2}^{\infty}\ell_j\right]\circ\beta\in \pi_n(\bbe_m)$ under $\Psi_n$. In contrast, it is not clear if $\left(\sum_{i=2}^{k}([a_1,a_i]\circ\beta)\right)_{k\in\bbn}\in \mcw_2$ lies in the image of $\Psi_{n}$. Indeed, one would like to form an infinite sum $\sum_{j=2}^{\infty}[\ell_1,\ell_j]$ and apply Hilton-Milnor invariants of $\beta$ in the finite-approximation levels but this infinite sum map is not continuous. The lack of obvious continuous representatives may be more analogous to the situation in dimension $n=m=1$ where the homomorphism $\Psi_{1}:\pi_1(\bbe_1)\to \check{\pi}_1(\bbe_1)$ fails to be surjective. Here, the analogous sequence of products of commutators \[(1,[a_1,a_2],[a_1,a_2][a_1,a_3],[a_1,a_2][a_1,a_3][a_1,a_4],\dots )\in \check{\pi}_1(\bbe_1)\] has no continuous representative in $\pi_1(\bbe_1)$. That surjectivity holds in the critical case $n=2m-1$ is then quite striking since it is the only situation where one can easily represent such growing sequences of commutators/Whitehead products. While this leaves us with an incomplete picture of the image of $\Psi_n$ (for $n>2m-1$), we identify a natural candidate.

For $i\in\bbn$, let $M_i$ be the collection of Hall words $w\in H_{\infty}$ such that
\begin{enumerate}
\item $L(w)\geq 2$,
\item $i$ is the smallest index for which $a_i$ appears as a letter in $w$. 
\end{enumerate}
Let $\mcg_{n}(m)=\prod_{i\in\bbn}\bigoplus_{w\in M_i}\pi_{n}(S^{(m-1)L(w)+1})$ and note that $\mcg_{n}(m)$ is naturally a subgroup of $\prod_{w\in H_{\infty}}\pi_n(S^{(m-1)L(w)+1})$. We represent an element of $\mcg_{n}(m)$ as a tuple $(\alpha_{i,w})_{i,w}$ for $\alpha_{i,w}\in \pi_{n}(S^{(m-1)L(w)+1})$ indexed over pairs $(i,w)$ with $i\in\bbn$ and $w\in M_i$.

\begin{proposition}\label{splittingprop}
There is a canonical monomorphism $\Theta:\mcg_{n}(m)\to \pi_n(\bbe_m)$ such that $ \Psi_{n}\circ\Theta(\mcg_{n}(m))=\phi_{n,\infty}(\mcg_{n}(m))$.
\end{proposition}

\begin{proof}
Define $\Theta:\mcg_{n}(m)\to \pi_{n}(\bbe_m)$ as follows. For Hall word $w$, let $\ell_{w}:S^{(m-1)L(w)+1}\to \bbe_m$ be the canonical map representing $w$ given by replacing each appearance of $a_i$ with $\ell_i$. Given $\alpha=(\alpha_{i,w})_{i,w}\in\mcg_{n}(m)$, note that the set $\supp_i(\alpha)=\{w\in M_i\mid \alpha_{i,w}\neq 0\}$ is finite for each $i$ (by definition of $\mcg_{n}(m)$). If $w\in \supp_i(\alpha)$, let $f_{i,w}:S^n\to S^{(m-1)L(w)+1}$ be a map representing $\alpha_{i,w}$ for each pair $(i,w)$. Define \[\Theta(\alpha)=\left[\sum_{i=1}^{\infty}\sum_{w\in \supp_{i}(\alpha)}\ell_{w}\circ f_{i,w}\right].\]

To see that $\Theta$ is well-defined, let $f_{i,w}'$ be another family of representatives for $\alpha_{i,w}$. For given pair $(i,w)$, we have $f_{i,w}\simeq f_{i,w}'$ and thus $\sum_{w\in \supp_{i}(\alpha)}\ell_w\circ f_{i,w}\simeq \sum_{w\in \supp_{i}(\alpha)}\ell_w\circ f_{i,w}$ in $X_{\geq i}$. This forms a sequential homotopy indexed over $i$ and thus $\sum_{i=1}^{\infty}\sum_{w\in \supp_{i}(\alpha)}\ell_{w}\circ f_{i,w}\simeq \sum_{i=1}^{\infty}\sum_{w\in \supp_{i}(\alpha)}\ell_{w}\circ f_{i,w}'$ by Proposition \ref{sumhomotopyprop}. 

To check that $\Theta$ is a homomorphism, let $\beta=(\beta_{i,w})_{i,w}\in\mcg_{n}(m)$ with representing maps $[g_{i,w}]=\beta_{i,w}$. Note that $\supp_{i}(\alpha+\beta)\subseteq \supp_i(\alpha)\cup \supp_i(\beta)$ where $F_i=(\supp_i(\alpha)\cup \supp_i(\beta))\backslash \supp_{i}(\alpha+\beta)$ is the finite (possibly empty) set of $w\in M_i$ such that $\alpha_{i,w}$ and $\beta_{i,w}$ are non-trivial and $\alpha_{i,w}=-\beta_{i,w}$. Note that $\sum_{i=1}^{\infty}\sum_{w\in F_i}\ell_w\circ (f_{i,w}+g_{i,w})$ is null-homotopic. Thus applying linearity of infinite sums to $A=\left[\sum_{i=1}^{\infty}\sum_{w\in F_i}(\ell_w\circ f_{i,w})+\sum_{i=1}^{\infty}\sum_{w\in F_i}(\ell_w\circ g_{i,w})\right]$, we have $A=0$. Then applications of Propositions \ref{linearity} and \ref{sumhomotopyprop} give that $\Theta(\alpha+\beta)$ is equal to:
\begin{eqnarray*}
 && \left[\sum_{i=1}^{\infty}\sum_{w\in \supp_{i}(\alpha+\beta)}\ell_{w}\circ (f_{i,w}+g_{i,w})\right]\\
&=& \left[\sum_{i=1}^{\infty}\sum_{w\in \supp_{i}(\alpha+\beta)}\ell_{w}\circ f_{i,w}+\sum_{w\in \supp_{i}(\alpha+\beta)}\ell_{w}\circ g_{i,w}\right]+A\\
&=& \left[\sum_{i=1}^{\infty}\left(\sum_{w\in \supp_{i}(\alpha+\beta)\cup F_i}\ell_{w}\circ f_{i,w}+\sum_{w\in \supp_{i}(\alpha+\beta)\cup F_i}\ell_{w}\circ g_{i,w}\right)\right]\\
&=& \left[\sum_{i=1}^{\infty}\left(\sum_{w\in \supp_{i}(\alpha)}\ell_{w}\circ f_{i,w}+\sum_{w\in \supp_{i}(\beta)}\ell_{w}\circ g_{i,w}\right)\right]\\
&=& \left[\sum_{i=1}^{\infty}\sum_{w\in \supp_{i}(\alpha)}\ell_{w}\circ f_{i,w}\right]+\left[\sum_{i=1}^{\infty}\sum_{w\in \supp_{i}(\beta)}\ell_{w}\circ g_{i,w}\right]\\
&=& \Theta(\alpha)+\Theta(\beta)
\end{eqnarray*}
For $\alpha\in \mcg_{n}(m)$ given above, we have $\phi_{n,\infty}(\alpha)=\left(\sum_{i=1}^{k}\sum_{w\in M_i\cap H_{n,k}}w\circ \alpha_{i,w}\right)_{k\in\bbn}$ and $\Psi_n(\Theta(\alpha))=\left(\left[\sum_{i=1}^{k}\sum_{w\in \supp_{i}(\alpha)\cap H_{n,k}}\ell_{w}\circ f_{i,w}\right]\right)_{k\in\bbn}$. Since these two are clearly equal, we have established the equality $ \Psi_{n}\circ\Theta(\mcg_{n}(m))=\phi_{n,\infty}(\mcg_{n}(m))$.
\end{proof}

\begin{proof}[Proof of Theorem \ref{thm3}]
The first summand $\prod_{\bbn}\pi_n(S^m)$ comes from the splitting of the weight-$1$ summand $\mcw_1\to \pi_n(\bbe_m)$ in Theorem \ref{exactlemma}. The second summand we construct as follows. Proposition \ref{splittingprop} and the bijectivity of $\phi_{n,\infty}$ imply that $H=\Theta(\mcg_{n}(m))$ is a subgroup of $\pi_n(\bbe_m)$ for which the restriction $\Psi_{n}|_{H}:H\to \phi_{n,\infty}(\mcg_{n}(m))$ is an isomorphism. To show that $H$ splits in a summand distinct from the weight-$1$ summand, we must show that $H\leq \ker(\sigma_{\#})\leq \pi_n(\bbe_m)$ where $\sigma:X\to \prod_{j\in\bbn}S^m$ is the inclusion. Let $\alpha=(\alpha_{i,w})\in \mcg_n(m)$ with $\alpha_{i,w}=[f_{i,w}]$. If $p_j:\bbe_m\to X_j$ is the retraction to the $j$-th sphere, then $p_j\circ \ell_{w}\circ f_{i,w}$ is null-homotopic for all $i,j\in\bbn$ and $w\in M_i$ (since $\ell_w$ is an iterated Whitehead product map determined by a Hall word $w$, which involves at least two distinct letters). Thus $[\ell_{w}\circ f_{i,w}]\in \ker(\sigma_{\#})$ for all pairs $(i,w)$ with $i\in \bbn$, $w\in M_i$. It is shown in \cite[Lemma 4.9]{BrazasWHProducts} that $\ker(\sigma_{\#})$ is ``closed under infinite sums" and it follows directly that $\Theta(\alpha)=\left[\sum_{i=1}^{\infty}\sum_{w\in \supp_{i}(\alpha)}\ell_{w}\circ f_{i,w}\right]\in \ker(\sigma_{\#})$. Hence, $H\leq \ker(\sigma_{\#})$.

Finally, it suffices to find an expression for $\mcg_{n}(m)$ that matches the expression in the statement of the theorem. For all $j\geq 2$, let $M_i(j)=\{w\in M_i\mid L(w)=j\}$ be the collection of Hall words of weight $j$ with $i$ being the smallest index for which $a_i$ appears in $w$ and note that $M_i(j)$ is countably infinite. Since $\{M_i(j)\}_{j\geq 2}$ is a partition of $M_i$ and $\pi_{n}(S^{(m-1)j+1})$ can only be non-trivial if $2\leq j\leq \frac{n-1}{m-1}$, we have
\begin{eqnarray*}
\mcg_{n}(m) & = & \prod_{i\in\bbn}\bigoplus_{w\in M_i}\pi_{n}(S^{(m-1)L(w)+1})\\
&\cong & \prod_{i\in\bbn}\left(\bigoplus_{2\leq j\leq \frac{n-1}{m-1}}\bigoplus_{w\in M_i(j)}\pi_{n}(S^{(m-1)j+1})\right)\\
&\cong & \bigoplus_{2\leq j\leq \frac{n-1}{m-1}}\left(\prod_{i\in\bbn}\bigoplus_{w\in M_i(j)}\pi_{n}(S^{(m-1)j+1})\right)
\end{eqnarray*}
where the last isomorphism is justified by the fact that finite direct sums commute with arbitrary direct products. By enumerating $M_i(j)$ and replacing this indexing set with $\bbn$, for each $i$, we match the form appearing in the statement of the theorem.
\end{proof}

\end{document}